\documentclass[11pt]{article}
\usepackage[english]{babel}

\usepackage{amssymb} 
\usepackage{mathtools}
\usepackage{amsthm} 
\usepackage{graphicx} 
\usepackage[all]{xy} 
\usepackage{wasysym}
\usepackage{wrapfig} 
\usepackage{xcolor} 
\usepackage{verbatim}
\usepackage{soul}

\usepackage[a4paper]{geometry}
\title{The homogeneous geometries of complex hyperbolic space}
\date{}

\author{J. L. Carmona Jim\'enez and
M. Castrill\'on L\'opez}

\numberwithin{equation}{section}

\renewcommand{\to}{\longrightarrow}

\newcommand{\R}{{\mathbb{R}}}
\newcommand{\C}{{\mathbb{C}}}

\newcommand{\T}{\nabla}
\newcommand{\TT}{\tilde{\T}}

\renewcommand{\S}{\mathrm{S}}
\newcommand{\U}{\mathrm{U}}
\newcommand{\SU}{\mathrm{SU}}

\newcommand{\Ad}{\mathrm{Ad}}

\newcommand{\Id}{\mathrm{Id}}

\newcommand{\s}{\mathfrak{s}}
\renewcommand{\u}{\mathfrak{u}}
\newcommand{\su}{\mathfrak{su}}
\newcommand{\f}{\mathfrak{f}}
\newcommand{\n}{\mathfrak{n}}
\renewcommand{\a}{\mathfrak{a}}
\renewcommand{\k}{\mathfrak{k}}
\newcommand{\h}{\mathfrak{h}}
\newcommand{\m}{\mathfrak{m}}
\newcommand{\g}{\mathfrak{g}}

\newcommand{\CH}{\mathbb{C}\mathrm{H}}

\newcommand{\curv}{\tilde{R}}
\newcommand{\tors}{\tilde{T}}

\newtheorem{proposition}{Proposition}[section]
\newtheorem{theorem}[proposition]{Theorem}

\newtheorem{corollary}[proposition]{Corollary}
\newtheorem{lemma}[proposition]{Lemma}
\newtheorem{remark}[proposition]{Remark}

\begin{document}

\maketitle

\begin{abstract}
	We describe the holonomy algebras of all canonical connections and their action on complex hyperbolic spaces $\mathbb{C}\mathrm{H}(n)$ in all dimensions ($n\in\mathbb{N}$).  This thorough investigation yields a formula for all K\"ahler homogeneous structures on complex hyperbolic spaces. Finally, we have related the belonging of the homogeneous structures to the different Tricerri and Vanhecke's (or Abbena and Garbiero's) orthogonal and irreducible $\mathrm{U}(n)$-submodules with concrete and determined expressions of the holonomy.
\end{abstract}

	\textbf{Key words.} Canonical connection, complex hyperbolic space, homogeneous structures, holonomy.

\section{Introduction} \label{Section 1}

Homogeneous manifolds provide a rich and varied class of spaces that has always deserved the special attentions of geometers. In this article, we focus in one of these spaces: the complex hyperbolic space $\CH(n) = \SU(n,1)/\S(\U(n)\U(1))$. This manifold plays an important role in different geometric situations (among  many others, the reader may have a look at the following recent works on $\CH(n)$: \cite{DDS2017}, \cite{SSS2020}, and \cite{W2018}) and in particular  it is a model for certain questions and classifications. For example, $\CH(n)$ is the paradigmatic space of the K\"ahler homogeneous structure tensors of linear type in Ambrose-Singer Theorem. More precisely, this theorem characterizes (local) homogeneity of Riemannian manifolds in terms of the existence of certain tensor satisfying a set of geometric partial differential equations. These tensors are classified in different algebraic classes, the dimension of some of them, the so-called linear classes, grow linearly with respect to the dimension of the manifold. The version of Ambrose-Singer Theorem for Kähler manifolds claims that a Kähler manifold with a homogeneous tensor of linear type is locally holomorphic isometric to the complex hyperbolic space, cf. \cite{GMM2000} (for the classical version of AS Theorem, see \cite{AS1958}; and for a recent generalized version of it, see \cite{CC2020}).

A homogeneous geometry of the complex hyperbolic space is understood as a transitive action by isometries of a Lie group $G$ on $\CH(n)$, that is, a homogeneous description $\CH(n) = G/H$, together with a canonical connection (the connections defined by a reductive decomposition, cf. \cite[Chapter X.2]{KN1963}) associated to it. Surprisingly, comprehensive lists of all homogeneous descriptions of homogeneous spaces are unknown in many cases.  Furthermore, even if all transitive actions on an homogeneous space are known, questions on the geometry of the canonical connections  are still unsolved in most of the cases. In  \cite{CGS2009}, a complete description of the groups $G$ acting transitively on the real or complex hyperbolic spaces are provided (so then, $\mathbb{R}\mathrm{H}(n)$ or $\mathbb{C}\mathrm{H}(n)$ are $G/H$ for certain subgroup $H$). With respect to the real case, the classification is completed by the analysis and characterization of the holonomies of all canonical connections in \cite{CGS2011}. Apart from that, there are only a few partial results on classical homogeneous manifolds (Berger 3-sphere \cite{GO2005}, or the 3-dimensional Heisenberg group \cite[Chapter 7]{TV1983}). With respect to the complex hyperbolic space, the only well known homogeneous geometry is the symmetric one, where one considers the full Lie group of isometries $G=\SU (n,1)$, and  the canonical connection coincides with the Levi-Civita connection.

The goal of this article is the characterization of the holonomy algebras of all canonical connections of the complex hyperbolic space $\mathbb{C}\mathrm{H}(n)$ as well as the classification of the homogeneous structure tensors associated to the reductive decompositions and homogeneous descriptions of these connections. This could be seen as a natural continuation of the job done for $\mathbb{R}\mathrm{H}(n)$ in \cite{CGS2011} to the Kähler case, within an ambitious program of analysis of homogeneous descriptions of manifolds. However, even though the starting lines of thought of the real case are similar in the complex case, the complete analysis will now require a more elaborate analysis as well as new ideas and approaches. This novelty is translated into the results. Actually, these new features are consistent with other the results in the literature, where the homogeneous structure tensors on Kähler manifolds (more visible in pseudo-Kähler manifolds) exhibit intrinsic features that we do not find in the purely (pseudo-)Riemannian setting (see, for example, \cite{BGO2011}). On the other hand, as a companion to the main result of the paper, we also get other geometric consequences.  For example, we prove that for any (non-symmetric) canonical connection always there exist two parallel vector fields.


This paper is organized as follows. In section \ref{Section 2}, we introduce the essential notions of Ambrose-Singer theorems, homogeneous structures, the descriptions of $\CH(n)$ as a homogeneous manifold and its K\"ahler structure. In section \ref{Section 3}, we prove a theorem that describes the holonomies of canonical connections on $\CH (n)$ in terms of algebraic maps. We also study how the holonomy algebras acts on $\CH (n)$ and we give a formula of its curvature and torsion tensors. We use these descriptions in section \ref{Section 4} to give a formula for any homogeneous structure on $\CH(n)$. In section \ref{Section 5}, we prove the main result of the article: we describe the holonomies of all canonical connections and its action on $\CH(n)$ in details. Finally, in section \ref{Section 6} we extend some partial results of \cite{CGS2009}. We study in which class (in the sense of Tricerri and Vanhecke, more precisely, Abbena and Garbiero \cite{AG1988}) each homogeneous structures is, and we characterize each (non-general) type of homogenous structure of complex hyperbolic space in terms of the concrete expression of the holonomy of the canonical connection.

\section{Preliminaries} \label{Section 2}

\subsection{The Ambrose-Singer equations}

Let $(M,g , J)$ be a connected, simply-connected and complete K\"ahler manifold of dimension $2n$ and let $S$ be a homogeneous K\"ahler structure tensor (see \cite{AG1988}), that is, a (1,2)-tensor field on $M$ such that
\begin{equation*}
	\TT R = 0, \quad \TT g = 0, \quad \TT J = 0, \quad \TT S=0,
\end{equation*}
where $\TT =  \T - S$, $\T$ is the Levi-Civita connection, and $R$ its curvature tensor. We fix a point $p\in M$. The holonomy algebra  $\mathfrak{hol}$ of $\TT$ is generated by the endomorphism $\curv_{XY}$ of $T_pM$, for all $X,Y\in T_pM$, where $\curv$ is the curvature of $\TT$. Note that these endomorphisms preserve the tensors $R$, $g$  and $J$. Furthermore, if we write $\m = T_pM$ for certain $p\in M$, by the so-called Nomizu construction (see \cite{N1954}), the vector space,
\begin{equation*}
	\tilde{\g} = \mathfrak{hol} + \m
\end{equation*}
can be endowed with a Lie bracket defined by
\begin{equation*}
	[U,V] = UV - VU, \quad [U,X] = U(X), \quad [X,Y] = \curv_{XY} -  \tors_X Y,
\end{equation*}
for $U, V \in \mathfrak{hol}$ and $X,Y \in \m$. Since, $(M,g,J)$ is connected, simply-connected and complete, we get a homogeneous description of $M = \tilde{G}/H$, where $\tilde{G}$ and $H$ are obtained by exponentiating $\tilde{\g}$ and $\mathfrak{hol}$ respectively. Under these conditions, the connection $\TT$ is the canonical connection associated to the reductive decomposition of $\tilde{\g} = \mathfrak{hol} + \m$ (see \cite[Vol. 2, p. 192]{KN1963}). Recall that, for any homogeneous space $M = G/H$ with reductive description $\g = \h + \m$, the canonical connection at $e=[H]$ is given by (\cite[p.20]{TV1983}),
\begin{equation} \label{Equation 2.1}
	\TT_B C = - [B,C]_{\m}
\end{equation}
for $B, C\in \m$, where the vector fields of the covariant derivative are regarded as the infinitesimal generators in $M$ induced by elements of $\m$. The canonical connection has the property that every left-invariant tensor on $M$ is parallel.

We now work from an infinitesimal point of view. Let $V = T_{p} M$. We will indistinctly work with $(1,2)$-tensors and $(0,3)$-tensors given by the isomorphism,
\begin{equation*}
(S_p)_{XYZ} = g((S_p)_X Y, Z), \qquad X,\,Y,\,Z\in V.
\end{equation*}
For the sake of convenience, $S_p$ will be also denoted simply as $S$. Condition $\TT _X g=S_X \cdot g=0$, is equivalent to the skew symmetry of the last two slots os the tensor above. In addition, condition $\TT _XJ=S_X\cdot J=0$ is equivalent to the invariance of the two slots with respect to $J$. We thus define
\begin{equation*}
\mathcal{K}(V) = \{S \in \otimes^3 V : \: S_{XYZ} = - S_{XZY}, S_{XJYJZ}=S_{XYZ}\}.
\end{equation*}
The group $\U(n)$ of unitary transformations of $V\simeq \mathbb{R}^{2n}$ acts on the space of tensors $\mathcal{K}(V)$. With respect to this action, $\mathcal{K}(V)$ can be decomposed in orthogonal and irreducible $\mathrm{U}(n)$-submodules (see \cite{AG1988, BGO2011}) as
\begin{equation}\label{Equation 2.2}
\mathcal{K}(V) = \mathcal{K}_1(V)\oplus \mathcal{K}_2(V)\oplus\mathcal{K}_3(V)\oplus\mathcal{K}_4(V),
\end{equation}
where
\begin{equation}\label{Equation 2.3}
	\begin{alignedat}{4}
\mathcal{K}_1 (V) &= \{S \in \mathcal{K}(V) : S_{XYZ} = \frac{1}{2} \left(S_{YZX} + S_{ZXY} + S_{JY JZ X} + S_{JZ XJY}  \right),\\ & \hspace{9cm} \mathrm{c}_{12} (S)= 0\}, \\
\mathcal{K}_2 (V) &= \{S \in \mathcal{K}(V) f: S_{XYZ} = \langle X,Y\rangle \theta_2(Z) - \langle X,Z\rangle \theta_2(Y) + \langle X, JY\rangle \theta_2(JZ)\\
&\hspace{3.1cm} - \langle X,JZ\rangle \theta_2(JY) -2\langle JY,Z\rangle \theta_2(JX), \theta_2 \in V^* \}, \\
\mathcal{K}_3 (V) &= \{S \in \mathcal{K}(V) : S_{XYZ} = - \frac{1}{2} \left(S_{YZX} + S_{ZXY} + S_{JY JZ X} + S_{JZ XJY}  \right) ,\\
&\hspace{9cm} \mathrm{c}_{12} (S)= 0\}, \\
\mathcal{K}_4 (V) &= \{S \in \mathcal{K}(V) : S_{XYZ} = \langle X,Y\rangle \theta_4(Z) - \langle X,Z\rangle \theta_4(Y) + \langle X, JY\rangle \theta_4(JZ)\\
& \hspace{3.1cm} - \langle X,JZ\rangle \theta_4(JY) +2\langle JY,Z\rangle \theta_4(JX), \theta_4 \in V^* \},
	\end{alignedat}
\end{equation}
and
\begin{align*}
\mathrm{c}_{12}(S)(Z)&=\sum _{i=1}^{2n} \left(S_{e_i e_i Z}\right),
\end{align*}
for any orthonormal basis $\{ e_1,\ldots, e_{2n}\}$ of $(V,g_p)$.

The philosophy behind this classification relies on its invariance: if $S$ belongs to one of these class at $p$, it will also belong to the same class in the classification at any other point $q\in M$. That is (cf. \cite[Chap. 4]{CC2019}), it is equivalent to study the symmetries of the homogeneous structure $S$ or those of the tensor element $S_p$.

\subsection{The complex hyperbolic space}\label{2.2}

We refer to \cite{G1999} for the basics on complex hyperbolic geometry that we outline now. Let $\hat{h}$ be the hermitian product in $\C^{n+1}$ defined by
$$
\hat{h}(X,Y) = \overline{Y_1} X_1 + ... +\overline{Y_{n-1}} X_{n-1} + \overline{Y_{n+1}} X_{n} + \overline{Y_{n}} X_{n+1}.
$$
The choice of this form (the so-called second hermitian form) instead of the canonical one (the first hermitian from) is compatible with our choice of notation of the algebra $\mathfrak{su}(n,1)$ given below. In any case, both forms are equivalent under a Cayley transformation. From $\hat{h}$ we define the Riemannian metric $\hat{g}(X,Y) = \mathrm{Re}(\hat{h}(X,Y))$ and the K\"ahler form $\hat{\omega} = \mathrm{Im}(\hat{h}(X,Y))$.

Let $\pi = \C^{n+1} \setminus \{0\} \to \C P^{n}$ be the canonical projection over the projective space. The complex hyperbolic space is defined as $\CH(n) = \pi (V_{-} )$, where $V_{-} = \{ x \in \C^{n+1} : \: \hat{h}(x,x) < 0 \}$. This is the Siegel domain model of $\CH (n)$. Unfortunately, this definition does not provide a canonical K\"ahler structure on $\CH(n)$. For that purpose, given $\mu>0$, we consider the anti-de Sitter space $$H^{2n+1}(\mu) = \{x \in \C^{n+1} :\: \hat{h}(x,x)= -\mu\}.$$ Obviously, $H^{2n+1}(\mu)$ is a embedded submanifold of $\C^{n+1}$ of dimension $2n+1$ such that $\pi (H^{2n+1}(\mu)) = \CH(n)$. The tangent space at $z\in H^{2n+1}(\mu)$ is $T_z H^{2n+1}(\mu)= \{X \in \C^{n+1} :\: \hat{g}(z,X) = 0\}$ and in particular the vector field $\xi_z = \frac{1}{\sqrt{\mu}}iz$ belongs to $T_z H^{2n+1}(\mu)$ for any $z\in H^{2n+1}(\mu)$. It is easy to check that the projection $\pi: H^{2n+1}(\mu) \to \CH(n)$ is an $S^1$-principal bundle the fibers of which are the integrable submanifolds of $\xi$. On the other hand, the vector field $\xi$ induces an orthogonal decomposition
$$
T_z H^{2n+1}(r) = T'_z H^{2n+1}(r) \oplus \R \cdot \xi_z,
$$
with $T'_z H^{2n+1}(\mu) = \{X \in \C^{n+1} : \: \hat{h}(X,z) = 0\}$. For any $\mu>0$, we equip $\CH(n)$ with a Riemannian metric $g$ and K\"ahler form $\omega$ from $\hat{g}$ and $\hat{\omega}$ respectively, by the pointwise isomorphism
\begin{equation}\label{Equation 2.4}
	\nu = \pi _*: T'_z H^{2n+1}(\mu) \to T_{\pi(z)}\CH(n),
\end{equation}
for all $z\in H^{2n+1}(\mu)$.

For any $\mu>0$, the Levi-Civita connection associated to the metric and complex structure induced by the projection $H^{2n+1}(\mu) \to \CH(n)$ has constant sectional holomorphic curvature equal to $-\frac{4}{\mu}$.

\subsection{The descriptions of $\CH (n)$ as a homogeneous manifold}
The description of $\CH(n)$ as a symmetric space is given by the quotient
\begin{equation*}
\CH(n) = \SU(n,1)/\S(\U(n)\U(1)),
\end{equation*}
where $\SU(n,1)$, the full set of isometries of the complex hyperbolic space, is the set of complex matrices of dimension $n+1$ preserving the form
$\mathrm{diag}(\Id_{n},-1)$, and determinant $+1$. We regard the group $\S(\U(n)\U(1))$ as the image of the monomorphism
\begin{equation} \label{Equation 2.5}
	\begin{split}
		\U(n) &\to  \SU (n,1) 			\\
	U &\mapsto \begin{pmatrix} U & 0 \\ 0 & \mathrm{det}(U)^{-1}\end{pmatrix}. 			
	\end{split}
\end{equation}
For the sake of simplicity of the computations in the rest of the article, it will be much more convenient to regard $\SU(n,1)$ as the set of complex matrices of dimension $n+1$, preserving the form $\mathrm{diag}(\Id_{n-1},
\big(\begin{smallmatrix}
	0 & 1 \\
	1 & 0
\end{smallmatrix}\big))$,
and determinant $+1$. With this choice, the Lie algebras of these groups are
\begin{equation*}
\mathfrak{su}(n,1)=
\begin{Bmatrix}
\begin{pmatrix}
B & v_1 & v_2 \\
-v_2^* & z & ib \\
-v_1^* & ia & -\overline{z} \\
\end{pmatrix} :\:
\begin{matrix}
z -\overline{z} + \mathrm{tr}(B) = 0; \\
B \in \u(n-1);\\
v_1, v_2 \in \C ^{n-1};\\
z \in \C;\; a,b \in \R
\end{matrix}
\end{Bmatrix},
\end{equation*}
\begin{equation*}
\mathfrak{k}=\mathfrak{s}(\mathfrak{u}(n)+\mathfrak{u}(1))=
\begin{Bmatrix}
\begin{pmatrix}
B & v & v \\
-v^* & i(a+b) & i(a-b) \\
-v^* & i(a-b) & i(a+b) \\
\end{pmatrix} :\:
\begin{matrix}
2i(a+b) + \mathrm{tr}(B) = 0; \\
B \in \u(n-1); \\
v \in \C^{n-1};\; a,b\in \R
\end{matrix}
\end{Bmatrix}.
\end{equation*}
As usual, the star $*$ stands for the transpose of the complex conjugate.
The reductive decomposition $$\mathfrak{su}(n,1)= \mathfrak{s}(\mathfrak{u}(n)+\mathfrak{u}(1)) + \mathfrak{m}$$ of this symmetric description is given by the Cartan decomposition defined by the involution
$\theta:\mathfrak{su}(n,1)\to \mathfrak{su}(n,1)$, $\theta (A)=\mathrm{diag}(\Id_{n-1},
\big(\begin{smallmatrix}
	0 & 1 \\
	1 & 0
\end{smallmatrix}\big)) \cdot A \cdot \mathrm{diag}(\Id_{n-1},
\big(\begin{smallmatrix}
	0 & 1 \\
	1 & 0
\end{smallmatrix}\big))$. The $(+1)$-eigenspace is $\mathfrak{k}=\mathfrak{s}(\mathfrak{u}(n)+\mathfrak{u}(1))$ whereas $\m$ is the $(-1)$-eigenspace is
\begin{equation}\label{Ad subspace}
	\m=
	\begin{Bmatrix}
		\begin{pmatrix}
			0 & v & -v \\
			v^* & a & ib \\
			-v^* & -ib & -a \\
		\end{pmatrix} :\:
		\begin{matrix}
			v \in \C^{n-1};\; a,b\in \R
		\end{matrix}
	\end{Bmatrix}.
\end{equation}
 
We now give the other homogeneous descriptions $\CH (n)=G/H$. First, we consider the Iwasawa decomposition $G=KAN$ as well as its infinitesimal version
\begin{equation*}
	\su(n,1) = \k + \a + \n,
\end{equation*}
where $\k=\mathfrak{s}(\mathfrak{u}(n)+\mathfrak{u}(1))$ is the compact part, $\a = \mathrm{span}_{\R} (A_0)$ is the unique maximal $\R$-diagonalizable subalgebra, $A_0 = \mathrm{diag} (0, \dots, 0,1,-1)$, and $\n = \n _1 + \n_2$ is the nilpotent part
\begin{equation*}
	\mathfrak{n_1}= 
	\begin{Bmatrix}
		\begin{pmatrix}
			0 & 0 & v \\
			-v^* & 0 & 0 \\
			0 & 0 & 0 \\
		\end{pmatrix}:\:
		\begin{matrix}
			v \in \C^{n-1}
		\end{matrix}
	\end{Bmatrix},
	\quad
	\mathfrak{n}_2= \R N_2,
	\quad
	N_2 =
	\begin{pmatrix}
		0 & 0 & 0 \\
		0 & 0 & i \\
		0 & 0 & 0 \\
	\end{pmatrix}.
\end{equation*} 
With respect to these last subspaces, they are the eigenspaces $\n_1 = \g_{\lambda}$, $\n _2 = \g_{2\lambda}$, associated to the set of roots $\Sigma =\{ \pm \lambda, \pm 2\lambda\}$, $\lambda (A_0)=1$.

Based on the result of Witte on cocompact Lie groups \cite{W1990}, one can determine (see \cite[pp. 568-569]{CGS2009}) all Lie groups acting transitively on $\CH(n)$:

\begin{theorem}\label{Theorem 2.1}
	The connected groups of isometries acting transitively on $\CH(n)$ are the full isometry group $\SU(n,1)$ and the groups $G = F_rN$, where $N$ is the nilpotent factor in the Iwasawa decomposition of $\SU(n,1)$ and $F_r$ is a connected closed subgroup of $\S(\U(n-1)\U(1))\R\subset \S(\U(n)\U(1))\R$ with non  trivial projection to $\R$.
\end{theorem}


In the following, we will repeatedly make use of the next brackets
\begin{equation}\label{corchetes}
	\begin{alignedat}{3}
		\text{for}\, X\in \n_1,\,\, [A_0, X] &= X,\quad		& [A_0, N_2] &= 2N_2,	\\
		[A_0, \s(\u(n-1)+ \u(1))] &= 0, \quad &[\n_1, \s(\u(n-1)+ \u(1))] &= \n_1,  \\
		[\n_2, \s(\u(n-1)+ \u(1))] &= 0, \quad	& [\n_1, \n_2] &= 0,	\\
		[\n_2, \n_2] &= 0, \quad		& [\n_1, \n_1]  &= \n_2,
	\end{alignedat}
\end{equation}
In particular, the last bracket reads
\begin{equation*}
	\footnotesize
	\begin{bmatrix}
		\begin{pmatrix}
			0 & 0 & v \\
			-v^* & 0 & 0 \\
			0 & 0 & 0 \\
		\end{pmatrix}
		,
		\begin{pmatrix}
			0 & 0 & w \\
			-w^* & 0 & 0 \\
			0 & 0 & 0 \\
		\end{pmatrix}
	\end{bmatrix}
	=
	\begin{pmatrix}
		0 & 0 & 0 \\
		0 & 0 & -v^*w + w^* v \\
		0 & 0 & 0 \\
	\end{pmatrix}
	=
	+2 \omega_0 (v, w) N_2
\end{equation*}
where $v,\,w\in \C^{n-1}$ and $\omega_0$ is the canonical symplectic (K\"ahler) form in $\C^{n-1}$.

\subsection{The K\"ahler structure of the homogeneous descriptions of $\CH(n)$}
\label{sec2.4}

Given a reductive decomposition $\g = \h + \m$, we identify the reductive complement $\m$ with $T_{\pi(z)} \CH(n)$ for some $z\in H^{2n+1}(r)$, and find the concrete expressions of the metric $g$ and the K\"ahler form $\omega$ in $\m$. Recall that, see \eqref{Equation 2.4}, the K\"ahler structure on $T_{\pi(z)}\CH(n)$ is induced by the following isomorphism
\begin{equation*}
	\nu = \pi _*: T'_z H^{2n+1}(\mu) \to T_{\pi(z)}\CH(n),
\end{equation*}
where $T'_z H^{2n+1}(\mu) = \{X \in \C^{n+1} : \: \hat{h}(X,z) = 0\}$. Therefore, given $X\in \m$, we first compute 
\[
\left. \frac{d}{dt}\right|_{t=0} \mathrm{exp(tX)}\cdot z= X\cdot z \in T_z H^{2n+1}(\mu),
\]
and then we orthogonally project to $T'_z H^{2n+1}(\mu)$, that is, 
\begin{equation*}
	\begin{alignedat}{2}
		\m &\to T'_z H^{2n+1}(\mu)  \\
		X &\longmapsto X\cdot z - \hat{h}(X \cdot z,z).
	\end{alignedat}
\end{equation*}

We choose $z= (0, ... , 0, \sqrt{\frac{\mu}{2}}-\sqrt{\frac{\mu}{2}}) \in H^{2n+1}(r)$.

For the symmetric description $\CH(n)=\SU(n,1)/\S(\U(n)\U(1)$ and the reductive decompositon $\g = \h + \m$ given in \eqref{Ad subspace}, we easily check that
\begin{equation*}
	X(a,b,v) =
	\begin{pmatrix}
		0 & v & -v \\
		v^* & a & ib \\
		-v^* & -ib & -a \\
	\end{pmatrix}
	\longmapsto
	\begin{pmatrix}
		2\sqrt{\frac{\mu}{2}} v \\
		a\sqrt{\frac{\mu}{2}} -ib\sqrt{\frac{\mu}{2}}  \\
		a\sqrt{\frac{\mu}{2}} - ib \sqrt{\frac{\mu}{2}}  \\
	\end{pmatrix} \in T'_z H^{2n+1}(\mu).
\end{equation*}

For the other descriptions, $G=F_rN$, $H=F_r\cap \S(\U(n-1)\U(1))$, the elements of the reductive complements $\m$ have always non-trivial projection to $\a + \n$ and are thus of the type
\begin{equation*}
	\label{m2}
	\begin{pmatrix}
		U & 0 & -2v \\
		2v^* & a+ic & 2ib \\
		0 & 0 & -a+ic \\
	\end{pmatrix} ;\:
	v \in \C^{n-1};\; a,b\in \R,
\end{equation*}
for certain $U\in \U (n-1)$. In this case we have
\begin{equation*}
	X(a,b,v) =
	\begin{pmatrix}
		U & 0 & -2v \\
		2v^* & a+ic & 2ib \\
		0 & 0 & -a+ic \\
	\end{pmatrix}
	\longmapsto
	\begin{pmatrix}
		2\sqrt{\frac{\mu}{2}} v \\
		a\sqrt{\frac{\mu}{2}} -ib\sqrt{\frac{\mu}{2}}  \\
		a\sqrt{\frac{\mu}{2}} - ib \sqrt{\frac{\mu}{2}}  \\
	\end{pmatrix} \in T'_z H^{2n+1}(\mu).
\end{equation*}

Hence, the pull-back of the metric and the symplectic form to $\m$ under this identification have the same expression for both cases $G=\SU(n,1)$ and $G=F_rN$. In particular, recalling that the metric and the symplectic form in $T_{\pi(Z)}\mathbb{C}H(n)$ are the projection by $\nu$ of the standard ones in  $T'_z H^{2n+1}(\mu)$, this pull-back forms are
\begin{equation}\label{metric}
	g(X_1(a_1,b_1, v_1),X_2(a_2,b_2,v_2))= \mu(a_1a_2 + b_1b_2+ 2 g_0(v_1,v_2))
\end{equation}
and
\begin{equation}\label{symplectic}
	\omega(X_1(a_1,b_1,v_1),X_2(a_2,b_2,v_2))= \mu(a_1b_2 - a_2b_1 + 2 \omega_0(v_1,v_2)),
\end{equation}
where $\omega _0$ and $g_0$ are the canonical symplectic (K\"ahler) and Riemannian metric on $\mathbb{C}^{n-1}$ respectively. Obvioulsy, the complex structure tensor $J$ on $\m$ is characterized by $\omega (X_1,X_2)=g(X_1,JX_2)$.

\color{black}


\section{The canonical connections on $\CH(n)$} \label{Section 3}

For the symmetric homogeneous description of $\C H(n)$, the canonical connection is the Levi-Civita connnection (the homogeneous structure tensor $S$ vanishes) so that its holonomy is the holonomy of the Riemannian manifold.

In the following, we confine ourselves to the non-symmetric descriptions $\mathbb{C}H(n)=G/H$ where $G = F_rN$ and  $H= F_r \cap \S(\U(n-1)\U(1))$. Since (see \eqref{Equation 2.5}) $\S(\U(n-1)\U(1)))\simeq\U(n-1)$ is compact and $F_r$ is closed, then $H$ is compact and, hence, reductive. Let $\h$ be the Lie algebra of $H$ and its reductive decomposition,
\begin{equation*}
	\h = \h_0 + \h_{ss}
\end{equation*}
where $\h_0$ is abelian and $\h_{ss}$ is semi-simple.

Let $\f_r$ be the Lie algebra of $F_r\subset \S(\U(n-1)\U(1))A$. With the restriction of the positive definite inner product $k(E,E') = \mathrm{Re}(\mathrm{tr}(E^*E'))$, $E, \, E' \in \s(\u(n-1)+ \u(1)) + \a$ to $\f_r$, we decompose
\begin{equation*}
	\f_r = \a _r + \h
\end{equation*}
where $\a_r$ is the orthogonal subspace to $\h$, that is $k(\a _r, \h)= 0$. By Theorem \ref{Theorem 2.1}, $\a_r$ projects to $\a$ and is of dimension 1, so that we can write it as $\a_r=\mathbb{R}A_r$ with $A_r = A_0 + H_r$ and $H_r \in \s ( \u(n-1) + \u(1)))$. Since $k(A_0, \h)= 0$ and $k(A_r, \h) = 0$, hence $k (H_r, \h) = 0$. Therefore, $H_r = 0$ or $H_r \notin \h$.

We claim that $[\a_r, \h] = 0$. On one hand, from the adjoint invariance of $k$, for every $H, \, H' \in \h$,
\begin{equation*}
	k([A_r, H], H') = -  k (A_r, [H', H]) = 0
\end{equation*}
and then $k([A_r, \h], \h) = 0$ which means that $[A_r, \h]$ belongs to $\a_r$. On the other hand, $[A_r, \h] = [A_0 + H_r, \h] = [H_r, \h] \subset \s(\u(n-1)+\u(1))$. Hence, $[A_r, \h]$ belongs to $\a_r \cap \s(\u(n-1)+\u(1)) = \{0\}$ because $\a_r$ projects non-trivially to $\a$.

Consequently, $\f_r$ has the following reductive decomposition,
\begin{equation*}
	\f_r = (\a_r + \h_0) + \h_{ss}.
\end{equation*}
In the expression
\begin{equation*}
	\g = \f_r + \n = (\a_r+ \h_0) + \h_{ss} + \n
\end{equation*}
we define
\begin{equation*}
	\s = \a + \n, \quad \s_r = \a_r + \n.
\end{equation*}
From \eqref{corchetes}, we have that $[\s,\s] = \n = [\s_r,\s_r]$.

Every canonical connection $\TT$ in $G/H$ is equivalent to a choice of an $\Ad(H)$-invariant subspace $\m$ complementary to $\h$ in $\g$. The subspace $\m$ can be regarded as a graph of an $\h$-equivariant map
\begin{equation*}
	\varphi_r : \s_r \to \h .
\end{equation*}
If we define the $\h$-equivariant map
\begin{equation*}
\chi_r : \s \to \s_r
\end{equation*}
extending the identity in $\n$ and mapping $A_0$ to $A_r$, we can consider
\begin{equation*}
\label{varphi}
	\varphi = \varphi_r \circ \chi : \s \to \h,
\end{equation*}
and $\m$ can be regarded as the image of
	\begin{eqnarray} \nonumber 
	\tilde{\cdot} \equiv \chi_{r} + \varphi: \a + \n &\to &\h\\
	\nonumber
	X&\longmapsto & \tilde{X}.
	\end{eqnarray}
	
\begin{lemma}\label{Lemma 3.1}
	The Lie algebra $\k$ of the holonomy group of the canonical connection $\TT$ associated to a reductive decomposition $\g = \h + \m$ is
	\begin{equation*}
		\k = \varphi_r (\n) = \varphi (\n).
	\end{equation*}
\end{lemma}

\begin{proof}
	The holonomy algebra $\mathfrak{hol}$ is generated by
	\begin{equation*}
		[\m, \m]_{\h}.
	\end{equation*}
		Let $H_0=\varphi(A_0)$, $H_{r0}=H_r+H_0$. Then
	\begin{align*}
		\tilde{A}_0 &= A_r +  H_0 = A_0 + H_{r0}, \\
		\tilde{N}_2 &=  N_2 + \varphi(N_2),
	\end{align*}
	by the $\h$-equivariance of $\varphi$ we have that $H_0$, $\varphi(N_2) \in \h_0$.
	
	The subspace $\m$ is spanned by $[\tilde{A}_0,\tilde{X}]$, $[\tilde{A}_0,\tilde{N}_2]$, $[\tilde N_2,\tilde X]$  and $[\tilde{X},\tilde{X}']$, for $X,X'\in\mathfrak{n}_1$. We study the projections to $\h$ of these brackets. Along the proof, we will repeatedly make use of \eqref{corchetes}.

 With respect to the first,
 \begin{equation}\label{AXtilde}
 	\begin{alignedat}{2}
 		[\tilde{A}_0, \tilde{X}] &= [A_0, X] + [A_0, \varphi(X)] + [H_{r0}, X] + [H_{r0}, \varphi(X)]  \\
 		&= X + 0 + [H_{r0}, X] + 0,
 	\end{alignedat}
 \end{equation}
	since $[H_r,\mathfrak{h}]=0$. As $H_{r0} \in \s(\u(n-1)+\u(1))$, then $[\tilde{A}_0, \tilde{X}]$ lies in $\n_1$. Furthermore, $H_{r0}$ acts on $\n_1$ as a skew-hermitian matrix acts in $\C^{n-1}$, and then $\mathrm{ad}_{H_{r0}}$ has no non-zero real roots. Therefore, the function
	\begin{equation*}
		f : = \Id + \mathrm{ad}_{H_{r0}}: \n_1 \to \n_1
	\end{equation*}
	is invertible. Hence $\{ [\tilde{A}_0, \tilde{X}] : \: X \in \n_1 \} = \n_1$. On the other hand, as $ X + \varphi(X)\in \mathfrak{m}$, for $X\in \mathfrak{n}_1$, then $(X)_{\h} = - \varphi(X)$. Consequently,
	\begin{equation*}
		\{ [\tilde{A}_0, \tilde{X}]_\mathfrak{h} : \: X \in \n_1 \} = \varphi ( \mathfrak{n}_1).
	\end{equation*}

    The second bracket yields
	\begin{align*}
		[\tilde{A}_0, \tilde{N}_2]
		&= [A_0, N_2] + [H_{r0}, N_2] + [A_0, \varphi(N_2)] + [ H_{r0}, \varphi(N_2)] = \\
		&= 2N_2 + 0 + 0 + 0,
	\end{align*}
	which means that $N_2 = \frac{1}{2} [\tilde{A}_0, \tilde{N}_2]$. Since $ N_2 + \varphi(N_2)\in \mathfrak{m}$, then $(N_2)_{\h} = - \varphi(N_2)$ and consequently,
	\begin{equation*}
		[\tilde{A}_0, \mathfrak{n}_2]_\mathfrak{h} = \varphi(\n_2).
	\end{equation*}

The third bracket is
	\begin{align*}
	[\tilde{N}_2, \tilde{X}]
	&= [N_2, X] + [\varphi(N_2), X] + [N_2, \varphi(X)] + [ \varphi(N_2), \varphi(X)]  \\
	&= 0 + [\varphi(N_2), X] + 0 + \varphi ([N_2, \varphi(X)])=[\varphi(N_2), X].
\end{align*}
As $[\varphi(N_2), X] \in \n_1$, then $[\tilde{N}_2, \tilde{X}]_\h = - \varphi([\varphi(N_2), X])= - [\varphi(N_2),\varphi( X)]= -\varphi([N_2,\varphi( X)])=0$.

Finally, for $X$, $X' \in \n_1$, the fourth bracket is
	\begin{equation}\label{XXtilde}
	[\tilde{X}, \tilde{X'}] = [X, X'] + [X, \varphi(X')] + [\varphi(X), X'] + [ \varphi(X), \varphi(X')].
	\end{equation}
	The first term lies in $\n_2$, the second and third lie in $\n_1$ and the fourth lies in $\h$. Hence, projecting to $\h$, we have
	\begin{align*}
	[\tilde{X}, \tilde{X'}]_{\h}
	&= - \varphi ([X,X']) - \varphi([X, \varphi(X')])  - \varphi([\varphi(X), X']) + [ \varphi(X), \varphi(X')] \\
	&=  - \varphi([X, X']) - [ \varphi(X), \varphi(X')].
	\end{align*}
	As $[\mathfrak{n}_1,\mathfrak{n_1}]=\mathfrak{n}_2$, and $[\varphi(\n_1), \varphi(\n_1)]\subset [\h, \varphi(\n_1)] = \varphi([\h, \n_1]) \subset \varphi(\n_1)$, then
	\begin{equation*}
		\{ [\tilde{X}, \tilde{X'}]_{\h}: X,X'\in\mathfrak{n}_1\} \subset \varphi (\mathfrak{n}_1) + \varphi ( \mathfrak{n}_2),
	\end{equation*}
and the proof is complete.
\end{proof}

From \cite[Vol. 2, Thm. 2.6]{KN1963}, we now get the expressions of the curvature and the torsion forms of all cannonical connections $\TT$ in $\CH(n)$.

\begin{corollary} \label{curvaturas tilde}
Following the notation above, the curvature form $\curv$ of a canonical connection $\TT$ of a non-symmetric description $\CH(n)=G/H$ is given by,
	\begin{align*}
	\curv_{\tilde{A}_0 \tilde{X}} &= - \varphi(f X),\quad \curv_{\tilde{A}_0 \tilde{N}_2 } = -2 \varphi(N_2), \quad \curv_{\tilde{N}_2 \tilde{X}} = 0\\
	\curv_{\tilde{X} \tilde{Y}} &= -2 \omega_0 (X, Y) \varphi(N_2) - [\varphi (X), \varphi (Y)]
	\end{align*}
	where $X,Y \in \n_1$ and $f : = \Id + \mathrm{ad}_{H_{r0}}$ as in the proof of Lemma \ref{Lemma 3.1}.
\end{corollary}

\begin{corollary}\label{torsiones tilde}
	Following the notation above, the torsion form $\tors$ of a canonical connection $\TT$ of a non-symmetric description $\CH(n)=G/H$ is given by,
	\begin{align*}
	\tors_{\tilde{A}_0} \tilde{X} &=  - \widetilde{fX}, \quad \tors_{\tilde{A}_0} \tilde{N}_2 = - 2 \tilde{N}_2, \quad \tors_{\tilde{N}_2} \tilde{X} = [\varphi(N_2), X]\\
	\tors_{\tilde{X}} \tilde{Y} &= -2 \omega_0 (X, Y) \tilde{N}_2 -  [\varphi (X), \tilde{Y}] - [\tilde{X}, \varphi(Y)]
	\end{align*}
	where $X,Y \in \n_1$ and $f : = \Id + \mathrm{ad}_{H_{r0}}$ as in the proof of Lemma \ref{Lemma 3.1}.
\end{corollary}

We also have the following result, that will be important in the study of homogeneous structures of linear type.

\begin{corollary}
	For any canonical connection of a non-symmetric homogeneous description $\CH(n)=G/H$, there exist two non-vanishing and non-collinear parallel vector fields.
\end{corollary}

\begin{proof}
	As $\curv_{B C} \tilde{A}_0 = 0$ and $\curv_{B C} \tilde{N}_2 = 0$, for any $B, C \in \m$, then we have that $\TT \tilde{A}_0 = 0$ and $\TT \tilde{N}_2 = 0$.
\end{proof}




\section{The homogeneous structure tensors of $\CH (n)$}\label{Section 4}

The goal of this section is to provide a technical but useful expression of the homogeneous structure tensor $S= \T-\TT$ for any reductive decomposition $\g = \h + \m$ of a (non-symmetric) homogeneous description $G/H=\mathbb{C}H(n)$.

First note that the decomposition 
$$
\m = \tilde{\a} + \tilde{\n}_1 + \tilde{\n}_2,
$$
given by the isomorphism $\tilde{\cdot} \equiv \chi_{r} + \varphi: \a + \n \to \m$ is orthogonal with respect to the metric \eqref{metric}. Furthermore, if we write
\begin{equation*}
B = \alpha_{B} \tilde{A}_0 + \eta_{B} \tilde{N}_2 + \tilde{X}_B,\qquad B\in \m,
\end{equation*}
the symplectic structure \eqref{symplectic} gives that
\begin{equation}\label{J}
\alpha_{JB} =  \frac{1}{2} \eta_B, \quad \eta_{JB} =  -2 \alpha_B, \quad \tilde{X}_{JB} = J \tilde{X}_B.
\end{equation}
For later convenience we define
\begin{equation*}
B' = \alpha_B \varphi (A_0)+ \eta_B \varphi (N_2) + \varphi (X_B),
\end{equation*}
\begin{equation*}
B_r = \widetilde{[H_r , X_B]},
\end{equation*}
where $H_r = \chi_r(A_0) = \tilde{A}_0 - \varphi(A_0)$.
Note that $H_r$ belongs to $\s(\u(n-1)+\u(1))$ so that it preserves $g$ and $\omega$.

\begin{lemma}\label{Lemma 4.1}
	For any $B,C, D \in \m$, we have the following formulas,
	\begin{equation} \label{Equation 4.2}
		\begin{split}
			[B,C]_{\m} &= \alpha_{B}(C + C_r - \eta_C \tilde{N}_2) - \alpha_C(B +  B_r - \eta_B \tilde{N}_2 )  \\
		& \quad + [B', C] -[C', B] +  \frac{4}{\mu} \omega(B, C) \tilde{N}_2
		\end{split}
	\end{equation}
	and
	\begin{equation} \label{Equation 4.3}
	\begin{split}
	g([B,C]_{\m}, D) &= \alpha_{B}g(C,D) - \alpha_Cg(B,D) +  \frac{\mu}{4}\eta_{D} \left( \alpha_C \eta_B - \eta_C \alpha_B  \right)\\
	&\quad + \alpha_{B} g(C_r,D) - \alpha_C g(B_r,D) + \eta_{D} \omega(B, C) \\
	&\quad + g([B', C],D) - g ([C', B],D).
	\end{split}
	\end{equation}
\end{lemma}

\begin{proof}
	From \eqref{AXtilde} and \eqref{XXtilde} we get
	\begin{align*}
	[B,C] &= \, \alpha_{B}\left( 2 \eta_C N_2 + X_C + [H_r,X_C]+ [H_0,X_C]\right)\\
	&\quad - \alpha_C \left( 2 \eta_B N_2 + X_B + [H_r,X_B]+ [H_0,X_B]\right) \\
	&\quad + [\tilde{X}_B, \tilde{X}_C] + [\varphi({X}_B), {X}_C] + [{X}_B, \varphi({X}_C)] +[\varphi({X}_B), \varphi({X}_C)] \\
	& \quad + \eta_B [\varphi(N_2), X_C] - \eta_C [\varphi(N_2), X_B].
	\end{align*}
	We add $\pm \alpha_{B} \alpha_C A_0$, $\pm \alpha _B \eta _C N_2$, $\pm \alpha _C \eta _B N_2$, in first and second rows and we project to $\m$. We get \eqref{Equation 4.2} from the expression \eqref{symplectic} that now looks like
    \[ 	2 (\alpha_B \eta_C - \eta_B \alpha_C)\tilde{N}_2 + \widetilde{[X_B,X_C]} = \frac{4}{\mu} \omega(B,C) \tilde{N}_2.
    \]
 Finally, \eqref{Equation 4.3} is a direct consequence.
\end{proof}

\begin{theorem}\label{Homogeneous structures}
	Following the notation above, the homogeneous tensor $S$ associated to a canonical connection $\TT$ reads
	\begin{equation}\label{Equation 4.4}
		\begin{split}
			g (S_BC, D) &=  \alpha_D g(B,C) -  \alpha_C g(B,D) +  \alpha_{JD}g(B,JC) - \alpha_{JC}g (B,JD) \\
			& -\alpha_{JB} \omega(\tilde{X}_C , \tilde{X}_D) + g([B', C], D)  + \alpha_B g( C_r , D)
		\end{split}
	\end{equation}
	for any $B,\, C, \,D \in \m$.
\end{theorem}

\begin{proof}
To compute the homogeneous structure tensor we use
\begin{equation*}
	2g (S_BC, D) = g([B,C]_{\m}, D) - g([C,D]_{\m},B) + g([D,B]_{\m},C)
\end{equation*}
for $B,\, C, \,D \in \m$, derived from \eqref{Equation 2.1} and \cite[p. 183]{B1987}.
	 Making use of Lemma \ref{Lemma 4.1} and taking into account that for any $U\in \s(\u(n-1)+\u(1))$ and  $\hat{B}, \hat{C} \in \s$, then $g([U, \hat{B}], \hat{C}) + g(\hat{B}, [U, \hat{C}]) = 0$, we have
	\begin{align*}
		2g (S_BC, D) &= 2 \alpha_D g(B,C) - 2 \alpha_C g(B,D) + \frac{\mu}{2}\eta_B( \alpha_C \eta_D - \alpha_D \eta_C)\\
		& \quad + \eta_{D} \omega(B, C) - \eta_{B}  \omega(C , D) +  \eta_{C}  \omega(D, B)\\
		& \quad+ 2g([B', C], D)  + 2 \alpha_B g( C_r , D),
	\end{align*}
	The proof is complete by \eqref{symplectic} and \eqref{J}.
\end{proof}

\color{black}


\section{The holonomy algebras on $\CH(n)$} \label{Section 5}

Let $\k$ be the holonomy algebra of a canonical connection on $\CH(n)$ different to $\su(n,1)$, that is $\k=\varphi (\n)$ for certain $\h$-equivariant morphism
\begin{equation*}
	\varphi|_{\n}: \a + \n \to \h .
\end{equation*}
We decompose $\n\simeq \mathbb{C}^{n-1}+\mathbb{R}$ in two orthogonal $\h$-modules,
\begin{equation} \label{orthogonal decomposition}
	\n = V_{\k} + V'
\end{equation}
where $V' \cong \ker \varphi|_{\n}$ and $V_{\k} = \{X\in \n : g(X,V') = 0\}\cong \n / \ker \varphi|_{\n}$. Since $\k$ is a subalgebra of $\s(\u(n-1)+\u(1))$, it is of compact type and reductive. We can decompose it as
\begin{equation*}
	\k = \k_0 + \k_{ss},
\end{equation*}
where $\k_0$ is abelian and $\k_{ss}$ is semi-simple. This gives a decomposition of $V_{\k}$ as
\begin{equation*}
	V_{\k} = V_0 +V_{ss}.
\end{equation*}

We now consider the complex tensor $J$ associated to the K\"ahler structure given above.

\begin{lemma}\label{Lemma 5.1}
	The holonomy algebra $\k$ acts trivially on $V_0 + J(V_0) + \n_2 + \a$.
\end{lemma}
\begin{proof}
	Since $\k \subset s(\u(n-1)+\u(1))$, from \eqref{corchetes} we have $[\k, \a] = 0$ and $[\k, \n_2] = 0$. On the other hand, for $N \in V_0$, we have $[\k, N]\subset V_0$ since $V_0$ is a $\k$-module. But $\varphi ([\k, N])=[\k, \varphi(N)] \subset [\k,\k_0]= 0$. Then $[\k, N] \in V'\cap V_0 = \{ 0 \}$. Finally, $\k$ preserves $J$, and we get the $[\k ,JN]=0$.
\end{proof}

\begin{lemma}\label{Lemma 5.2}
	The Lie algebra $\k_0$ acts trivially on $V_{ss} + J(V_{ss})$.
\end{lemma}
\begin{proof}
	We consider $N \in V_{ss}$, then, $[\k_0, N] \in V_{ss} \cap V' = \{ 0 \}$.
\end{proof}

\begin{lemma} \label{Lemma 5.3}
There exists a $\k$-module $W$ isomorphic to  $V_{ss}$ and contained in $V'$.
\end{lemma}
\begin{proof}
	Let $\k_{ss} = \k_{s_1} + ... + \k_{s_l}$ be the decomposition of $\k_{ss}$ in sum of simple Lie algebras and let $V_{s_i}$ be the $\k$-module associated to $\k_{s_i}$ for any $i \in \{1, ... , l\}$.
	
	First we proof that, for any $i \in \{1, ..., l\}$ and any non zero $X\in V_{s_i}$, then $J X \notin V_{ss}$. Suppose that $JX \in V_{ss}$. The sum decomposition gives that $[\k_{s_j} , X]=\delta_{ij}V_i$, for any $j$, where $\delta_{ij}$ is the Kronecker delta. Then, $[\k_{s_j}, JX] = \delta_{ij} JV_{s_i}$ and, since $JX \in V_{ss}$, we have $JX = (JX)_1 + ... + (JX)_l$ with $JX_j\in V_j$. Therefore,  we must have $JV_{s_i}=V_{s_i}$.	Let $X_B$ and $X_C$ be two vectors in $V_{s_i}$. On one hand, as $\varphi([\tilde{X}_B',X_C]- [X_B,\tilde{X}_C']) = 0$ and  $\mathrm{ker}(\varphi|_{V_{ss}}) = \{0\}$, we have $[\tilde{X}_B',X_C] = - [\tilde{X}_C',X_B]$. On the other, if we consider $X_B, X_C, X_D \in V_{s_i}$,
	\begin{align*}
		g([(\tilde{X}_B)',X_C], X_D) & = 	g([(\tilde{X}_B)',JX_C], JX_D)
	\end{align*}
	as $J X_C \in V_{s_i}$,
	\begin{align*}
		g([(\tilde{X}_B)',JX_C], JX_D) & = -	g([(J\tilde{X}_C)', X_B], JX_D) =  +	g([(J\tilde{X}_C)', JX_D], X_B)\\
		& =  -	g([(J\tilde{X}_C)', X_D], JX_B) = + g([(\tilde{X}_D)', JX_C], JX_B) \\
		&= g([(\tilde{X}_D)', X_C], X_B) = - g([(\tilde{X}_B)', X_C], X_D).
	\end{align*}
	 Thus, $g([(\tilde{X}_B)',X_C], X_D) = - g([(\tilde{X}_B)',X_C], X_D) = 0$ for every $X_B, X_C, X_D \in V_{s_i}$ and then $[(\tilde{X}_B)',X_C]=0$. As $[\k_{s_j} , X]=\delta_{ij}V_i$, $\k_{s_i}$ acts effectively on $V_{s_i}$ and trivially on $V_{s_j}$ with $j\neq i$ if $[\tilde{X}_B',X_C] = 0$ for every $X_B, X_C \in V_{s_i}$ then $\k_{s_i}$ is necessarily zero, and this is a contradiction with the fact that $0\neq X\in \k _{s_i}$.
	
	Now, we can say that, for any $X \in V_{s_i}$, $X\neq 0$, the decomposition of $\k$-modules $V = V_{ss} + V_0 + V'$, gives $JX = (JX)_{ss} + (JX)_0 + (JX)_{V'}$ with $(JX)_0 + (JX)_{V'}\neq 0$. But, as $\k_{s_i} \subset \k_{ss}$ acts trivially on $V_0$ (Lemma \ref{Lemma 5.1}),  $JV_{s_i}  = [\k_{s_i}, JX] = [\k_{s_i},(JX)_{ss}] + [\k_{s_i}, (JX)_{V'}] \subset V_{ss} + V'$, we get $(JX)_0 = 0$. We thus define the $\k$-module $V_{s_i}' = [\k_{s_i}, (JX)_{V'}] \subset V'$, with $\dim (\k_{s_i}) = \dim (V_{s_i})  \geq \dim (V_{s_i}') $. The map $\psi_i: V_{s_i} \to V_{s_i}'$, $\psi_i(X) = (JX)_{V'}$ is a $\k$-module isomorphism, injective, because of $\ker \psi_i = \{0\}$, and surjective, because of dimensions.
	
	Finally, we consider the $\k$-module morphism $\psi: V_{ss} \to V'$, $\psi(X)= (JX)_{V'}$. If $X =X_1 + ... + X_l$, $X_i\in \k_{s_i}$, with $\psi(X) = 0$, then $0 = [\k_{s_i}, \psi (X)] = \psi([\k_{s_i}, X]) = \psi([\k_{s_i}, X_i]) = \psi_i ([\k_{s_i},X_i])$. Since $\psi_i$ is injective, we have $[\k_{s_i}, X_i] = 0$ and then $X_i=0$ because $\k_{s_i}$ acts effectively on $V_{s_i}$, for all $i$. Then $\psi$ is injective and the image $\psi (V_{ss})$ is the $W$ of the statement.
\end{proof}
	


\begin{theorem} \label{Big theorem}
	The holonomy algebras of canonical connection on $\CH (n)$, are $\s\u(n,1)$ and all reductive Lie algebras of compact type
	\begin{equation*}
		\k = \k_0 + \k_{ss}
	\end{equation*}
	with $\k_0 \cong \C^r \times  \R^s$ abelian where $s \geq 0$, and $r \geq 0$, satisfying any of the following two constraints of dimensions,
	\begin{equation*}
		3r + 2s + \dim(\k_{ss}) \leq n-1,
	\end{equation*}
	or, $s \geq 1$ and
	\begin{equation*}
		3r + 2(s-1) + 1 + \dim (\k_{ss}) \leq n-1.
	\end{equation*}
\end{theorem}

\begin{remark}
As we will see in the proof, the conditions and structure of $\k$ provided by this Theorem are determined by the role played by $\n = \n_1 + \n_2$ with respect to the morphism $\varphi$ of the equality $\k = \varphi (\n)$. 

First, the existence of two different conditions on the dimensions shows whether $\varphi (N_2)\in \R ^s$ vanishes or not. Of course, there are many choices of $r,s$ and $\mathrm{dim}\k_ss$ that satisfy both inequalities. In these cases, the two possibilities $\varphi (N_2)=0$ and $\varphi (N_2)\neq 0$ can be considered.

On the other hand, the splitting of $\k _0$ into complex and real parts exhibits the action of $J$. In particular $\C ^r$ is the $J$-invariant part of $(\mathrm{ker}\varphi )^\perp \subset \n_1\simeq \C ^{n-1}$ from the identification $\k \simeq (\mathrm{ker}\varphi )^\perp$.
\end{remark}

\begin{proof}

Let $\TT$ be a canonical connection on $\CH(n)$ (different to the symmetric one) and $\curv$ its curvature tensor. As $\k_0 \cong \C^r \times \R^s$ is Abelian, then an effective metric representation of $\k_0$ is at least of real dimension $4r + 2s$. On the other hand, $\k_{ss}$ acts effectively on $V_{ss} + W$ where $W$ is defined in Lemma \ref{Lemma 5.3}, and $\dim(V_{ss} + W)=2 \dim(\k_{ss})$. We know that $\k$ acts effectively on $\a +\n$ and the action is trivial on $V_0 + J(V_0) + \n_2 + \a$  (Lemma \ref{Lemma 5.1}). Since the actions of $\k_0$ and $\k_{ss}$ are inequivalent, we have
\[
4r+2s+2\dim(\k _{ss})\leq \dim(\a+\n)-\dim(V_0 + J(V_0) + \n_2 + \a).
\]
We distinguish two cases, providing the two inequalities of the statement:
\begin{itemize}
\item If we have $\curv_{\tilde{A}_0 \, \tilde{N}_2} = -2\varphi(N_2)=0$, then $N_2\notin V_0$ and $N_2 \notin J V_0$ (since $2 JN_2=A_0\notin V_0 \subset V_{\k}$). Then
\[
\dim(V_0 + J(V_0) + \n_2 + \a)=2\dim(\C^r+\R^s)-\dim(V_0\cap JV_0)+2=2(2r+s)-2r+2
\]
so that
\[
6r+4s+2\dim(\k _{ss})\leq 2n-2,
\]
and we get the first inequality.

\item If we have $-2\varphi(N_2)=\curv_{\tilde{A}_0 \, \tilde{N}_2} \neq 0$, as $[\n_2, \s(\u(n-1)+ \u(1))] = 0$ (see \eqref{corchetes}) and $\k \subset \s(\u(n-1)+ \u(1))$, we have that $\varphi(N_2) \in \k_0$ or $N_2\in V_0$, and in fact, $N_2\in \mathbb{R}^s$. Furthermore $A_0\in JV_0$. Then
\[
\dim(V_0 + J(V_0) + \n_2 + \a)=2\dim(\C^r+\R^s)-\dim(V_0\cap JV_0)=2(2r+s)-2r
\]
so that
\[
6r+4s+2\dim(\k _{ss})\leq 2n,
\]
and the second inequality is proved.
\end{itemize}




Conversely, we will start from a Lie algebra $\k$ satisfying the conditions of the theorem above, and we will construct and $\k$-equivariant map $\varphi$ as in Lemma \ref{Lemma 3.1} which is in direct correspondence with a reductive decomposition $\g = \k + \m$ and with a canonical connection $\TT$ of $\CH(n)$.

First, let $\k = \k _0 + \k_{ss}$ be a reductive Lie algebra of compact type which satisfy the first inequality of dimensions given in the theorem.

Since $\mathrm{dim}\,\k <\mathrm{dim}\,\n_1$ we can consider a subspace $V_{\k} \subset \n_1$ together with a linear $\k$-modules isomorphims $\psi: V_{\k} \to  \k$. We transfer the decomposition $\k=\k_0+\k _{ss}$ to $V_{\k}$ by $\psi$ as $V_0 + V_{ss}$. We consider $V_1$ a minimal effective metric and complex representation of $\k_0 \cong \C^r \times \R^s$ and we take $W$ a copy of $V_{ss}$. As $\dim (V_1) = 4r + 2s$, and using the constraints of dimensions we define,
\begin{equation*}
	\n_1 = \C^{n-1} = V_{\k} + V_1 + W+ \R^s + \C^m
\end{equation*}
for certain $m\geq0$, where $ L =  \R^{s} + \C^{m}$ is a trivial submodule of $\k$, and we define a complex structure between the $\R^{s}$ factors of $V_0$ and $L$. This decomposition of $\n_1$ admits a $\k$-invariant inner product compatible with the complex structure: in $V_0 + L = \C^{r+s+m}$; in $V_1$, the inner product is induced by the effective representation; in $V_{ss}$, we use that it is isomorphic to a semi-simple Lie subalgebra of a compact Lie algebra. Hence, it admits a bi-invariant real inner product which can be extended to a hermitian inner product in $V_{ss} + W$. Therefore, $\k$ act effectively in $\C^{n-1}$ preserving the hermitian inner product and $\k$ arises as a Lie subalgebra of $\u(n-1)$ which is isomorphic to $\s(\u(n-1) + \u(1))$. Finally, we define $\varphi$ to be $\psi$ on $V_{\k}$ and zero on $V_1 + W+ L + \a + \n_2$. This realises $\k$ as the holonomy algebra of a canonical connection on $\CH(n)$ with $\g = \k + \a +\n$.

Second, let $\k = \k _0 + \k_{ss}$ be a reductive Lie algebra of compact type which satisfying the second inquality of dimensions given in the theorem and $s \geq 1$.

We consider a copy $V_{\k}\subset \n$, with non-trivial projection to $n_2$, of the $\k$-module $\k$ given by a linear isomorphism $\psi: V_{\k} \to  \k.$ We decompose $V_{\k}$ as $V_0 + V_{ss}$ by $\psi$, with $V_0 = \n_2 + V_0'$. We consider $V_1$ a minimal effective metric and complex representation of $\k_0 \cong \C^r \times \R^s$ and we take $W$ a copy of $V_{ss}$. As $\dim (V_1) = 4r + 2s$, and using the constraints of dimensions we define,
\begin{equation*}
	\n_1 = \C^{n-1} = V_0' + V_{ss} + V_1 + W+ \R^s + \C^m,
\end{equation*}
for certain $m\geq 0$, where $ L = \C^{m}+ \R^{s-1}$ is a trivial submodule of $\k$ and we define a complex structure between the $\R^{s}$ factors of $V_0'$ and $L$. This decomposition of $\n_1$ admits a $\k$-invariant inner product compatible with the complex structure: In $V_0' + L = \C^{r+(s-1)+m}$ we choose any inner product compatible;  in $V_1$ the inner product is induced by the effective representation; in $V_{ss}$ we use that it is isomorphic to a semi-simple Lie subalgebra of a compact Lie algebra, so that, it admits a bi-invariant real inner product which can be extended to a hermitian inner product in $V_{ss} + W$. Hence, a $\k$ arises as a Lie subalgebra of $\u(n-1)$ which is isomorphic to $\s(\u(n-1) + \u(1))$. Finally, we define $\varphi$ to be $\psi$ on $V_{\k}$ and zero on $V_1 + L + W + \a$ then realises $\k$ as the holonomy algebra of a canonical connection on $\CH(n)$ with $\g = \k + \a +\n$.

	The construction of the canonical connection associated to $\k$ is in the case $A_r = A_0$ ($\a_r = \a$, see Section \ref{Section 3}) the construction is analogous for $A_r = A_0 + H_r$ with $H_r \notin \k$. Also, note that the Lie algebra $\k$ exponentiates to a closed (so compact) subgroup $K$ of $S(U(n-1)U(1))$.
\end{proof}



\section{The description of homogeneous types} \label{Section 6}

To complete the study of homogeneous structures in the complex hyperbolic space we characterize, in terms of the holonomy algebra of the canonical connection, when a homogeneous structure belongs to the different subspaces of the decomposition \eqref{Equation 2.3}
\begin{equation*}
	\mathcal{K}(V) = \mathcal{K}_1(V)+\mathcal{K}_2(V)+\mathcal{K}_3(V)+\mathcal{K}_4(V).
\end{equation*}
For convenience in the rest of the section, the direct sums $\mathcal{K}_i(V)+\mathcal{K}_j(V)$ and $\mathcal{K}_i(V)+\mathcal{K}_j(V)+\mathcal{K}_k(V)$ will be denoted by $\mathcal{K}_{i+j}(V)$ and $\mathcal{K}_{i+j+k}(V)$ respectively.

In Theorem \ref{Homogeneous structures}, we found a global expression for every homogeneous structure in $\CH(n)$. We split the expression of \eqref{Equation 4.4} in four tensors,
\begin{align*}
	E_1 (B,C,D) &= \alpha_D g(B,C) - \alpha_C g(B,D) +  \alpha_{JD}g(B,JC) - \alpha_{JC}g (B,JD),\\
	E_2 (B,C,D) &= \alpha_{JB} g(\tilde{X}_{JC}, \tilde{X}_D),\\
	E_3 (B,C,D) &= g([B', C], D),\\
	E_4 (B,C,D) &= \alpha_B g( C_r , D).
\end{align*}
 
\begin{lemma}
	The tensor element $E_1$ belongs to $\mathcal{K}_{2+4}(V)$ with $$\theta_2 = \theta_4 = \frac{g(\tilde{A}_0, \cdot)}{g(\tilde{A}_0,\tilde{A}_0)}.$$ 
\end{lemma}

\begin{lemma}\label{Lemma 6.2}
	The tensors elements $E_2, E_3, E_4$ belongs to $\ker(c_{12})$.
\end{lemma}

\begin{proof}
	We consider an orthonormal basis
	\begin{equation*}
		\mathcal{B}_1 = \{\tilde{e}_i, \tilde{e}_n = \frac{\tilde{A}_0}{\sqrt{\mu}}, \tilde{e}_{i+n}, \tilde{e}_{2n} = \frac{-2}{\sqrt{\mu}}\tilde{N}_2 :\: i=1,...,n-1 \}
	\end{equation*}
	of $\m$ such that $ J \tilde{e}_i = \tilde{e}_{i+n}$ and $J \tilde{e}_{i+n} = - \tilde{e}_i$. We check directly $\mathrm{c}_{12} (E_2) = 0$ and $\mathrm{c}_{12} (E_4) = 0$. We  proof that $\mathrm{c}_{12} (E_3) = 0$. First, note that $[\varphi(X), \tilde{X}] = [\varphi(X), X]$, for any element  $\tilde{X} \in \m$. Secondly, we consider the orthogonal decomposition $\n = V_{\k} + V'$ \eqref{orthogonal decomposition} and $\m = \tilde{\a} + \tilde{V}_{\k} + \tilde{V}'$. Therefore, we take a basis $\mathcal{B}_2 = \{\frac{\tilde{A}_0}{\sqrt{\mu}}, \tilde{e}_j :\: j= 2, ... ,2n\}$ preserving the decomposition above, that is, there exists $k\in \{2, ..., 2n\}$ such that $\tilde{e}_j \in \tilde{V}_{\k}$ for any $j \in \{2, ..., k\}$ and $\tilde{e}_j \in \tilde{V}'$ for any $j \in \{k+1, ..., 2n\}$. From $[\varphi(A_0), \tilde{A}_0] = 0$: if $X \in V_{\k}$, then it is satisfied that $[\varphi(X), X] \in V_{\k} \cap V' = \{0\}$, and if $X \in V'$, then $\varphi(X) = 0$. Therefore, each term of the sum $\mathrm{c}_{12} (E_3) = g( [(\tilde{e_1})', e_1], D) +  \sum_{j=1}^{k} g( [(\tilde{e_j})', e_j], D)  + \sum_{j=k+1}^{2n} g( [(\tilde{e_j})', e_j], D)$ vanishes.
\end{proof}

\begin{corollary}
	Given any homogeneous K\"ahler structure $S$ on $\CH (n)$ we have that $$\mathrm{c}_{12} (S) (D)= 4n\, \alpha_D.$$
\end{corollary}

\begin{corollary} \label{Corollary 6.4}
	The tensor element $E_2$ is of type $\mathcal{K}_{2+3+4}(V) = \mathcal{K}_{2}(V) + \mathcal{K}_{3}(V)+ \mathcal{K}_{4}(V)$.
\end{corollary}
\begin{proof}
	The homogeneous structure $E_2 =  \alpha_{JB} \left( g (JC,D) - \left( \alpha_{JC} \alpha_D  + \eta_D \eta_{JC}  \right) \right)$, the first term lies in $\mathcal{K}_{2+4}(V)$ with $\theta_2 = - \theta_4 = \alpha$ and the second term lies in $\mathcal{K}_{3+4}(V)$.
\end{proof}


We recall the inner product in $\mathcal{K}(V)$ defined as
\begin{equation*}
	\langle S, S' \rangle = \sum_{i,j,k=1}^{2n}  S_{\tilde{e}_i \tilde{e}_j \tilde{e}_k} S'_{\tilde{e}_i \tilde{e}_j \tilde{e}_k}
\end{equation*}
where $\mathcal{B} = \{\tilde{e}_1, ... , \tilde{e}_{2n}\}$ is any orthonormal basis of $(V, g)$. With respect to this inner product we have
\begin{align*}
	\ker(c_{12})^{\perp} &= \{ S \in \mathcal{K}(V) :\: S_{BCD} = g(B,C)\sigma_2(D) -  g(B,D) \sigma_2(C) \\
	& \qquad \qquad +  g(B,JC) \sigma_2(JD) - g (B,JD)\sigma_2(JC),\, \sigma_2 \in V^*\}.
\end{align*}

\begin{lemma}\label{lemma 6.5}
	Let $S$ be a K\"ahler homogeneous structure on $\CH (n) = G/H$  and $\mathfrak{hol}$ its holonomy algebra. Then, $S$ belongs to $\ker(c_{12})^{\perp}$ if and only if $\mathfrak{hol}$ is one dimensional with $\curv_{\tilde{A}_0 \tilde{N}_2} \tilde{X}=  J \tilde{X}$, for $\tilde{X} \in \tilde{\n}_1$ and $\a_r = \a \subset \ker \varphi$.
\end{lemma}

\begin{proof}
	Since $E_1 \in \ker(c_{12})^{\perp}$, from Lemma \ref{Lemma 6.2} we have that $S \in \ker(c_{12})^{\perp}$ if and only if $E_2 + E_3 + E_4 = 0$, that is, for any $B,C,D\in \m$,
	\begin{equation}\label{Equation 6.1}
		g([B', C], D) + \alpha_B g( C_r , D) =- \alpha_{JB} g(\tilde{X}_{JC}, \tilde{X}_D).
	\end{equation}
	
	Suppose $E_3 + E_4 = - E_2$. Then, if we take $B = \tilde{X}_B \in \tilde{\n}_1$, we get $g([\varphi(X_B), C], D) = 0$, which means that $\varphi(X_B) = 0$ and $\mathfrak{hol}= \varphi (\n)$ is only generated by $\varphi (N_2)$. If we take $B = \tilde{A}_0$, then $  g([\varphi(A_0), C], D) +  g( [H_r, C] , D) = 0$ which means that $ \varphi (A_0) = - H_r$. Nevertheless, in the beginning of Section \ref{Section 3} we showed, $\varphi(A_0) \in \h$ and $ H_r \notin \h$ or $H_r = 0$ where $\h$ is the Lie algebra of $H$. Then, necessarily $H_r = 0 = \varphi(A_0)$, indeed $\a_r = \a \subset \ker \varphi$. Finally, if we take $B = \tilde{N}_2$ in \eqref{Equation 6.1} we have
	\begin{align*}
		\eta_B g([\varphi (N_2), C], D) &= - \alpha_{JB} g(\tilde{X}_{JC}, \tilde{X}_D)
	\end{align*}
	and as $ \varphi (N_2)$ acts trivially on $\a + \n_2$ and $\eta_B = 2 \alpha_{JB}$, then,
	$$
	\eta_B g([\varphi (N_2), C], D) = 2\alpha_{JB} g([\varphi (N_2), \tilde{X}_C], \tilde{X}_D) = - \alpha_{JB} g(\tilde{X}_{JC}, \tilde{X}_D) .
	$$
	Therefore, 	$ \tilde{X}_{JC} = - 2[\varphi (N_2), \tilde{X}_C] = \curv_{\tilde{A}_0 \tilde{N}_2} \tilde{X}_C$.
	
	Conversely, it is easy to check that for that holonomy algebra conditions it is satisfied $E_2 + E_3 + E_4 = 0$.
\end{proof}

\begin{lemma}\label{lemma 6.6}
	Let $S= E_1 + E_2 + E_3 + E_4$ be a K\"ahler homogeneous structure on $\CH (n) = G/H$  and $\mathfrak{hol}$ its holonomy algebra. Then, $E_3+E_4$ belongs to $\mathcal{L}(V) =\{ S \in \mathcal{K}(V) :\: S_{BCD} =  g(\tilde{X}_{JC},\, \tilde{X}_D) \phi (JB) ,\, \phi \in V^*\}$ if and only if happens one of the following two:
	\begin{itemize}
		\item The holonomy algebra  $\mathfrak{hol}$ is one dimensional such that $ \curv_{\tilde{A}_0 \tilde{N}_2} \tilde{X}=  \lambda J \tilde{X}$, for $\tilde{X}\in \tilde{\n}_1$ where $\lambda \in \R$ with $\lambda \neq 1$, $\a_r = \a$ and $\varphi (\a) = \varphi(\n_2) = \mathfrak{hol}$.
		\item The holonomy algebra  $\mathfrak{hol}$ is trivial and $[\varphi(A_0)+ H_r, \tilde{X}] = \beta J\tilde{X}$, for $\tilde{X} \in \tilde{\n}_1$ where $H_r= A_r - A_0$ and $\beta \in \R$.
	\end{itemize}
\end{lemma}
\begin{proof}
	Suppose that $E_3 + E_4 \in \mathcal{L}$, that is, for any $B,C,D\in \m$,
	\begin{equation}\label{Equation 6.2}
		g([B', C], D) + \alpha_B g( C_r , D) = \phi(B) g(\tilde{X}_{JC},\tilde{X}_D).
	\end{equation}
	If we take $B=\tilde{X}_B\in \tilde{\n}_1$, then,
	\begin{equation*}
		g([\varphi(X_B), C], D) = g([\varphi(X_B), \tilde{X}_C], \tilde{X}_D) = \phi(\tilde{X}_B) g(\tilde{X}_{JC}, \tilde{X}_D),
	\end{equation*}
	where for the first equality we have used that $\mathfrak{hol}$ acts trivially in $\tilde{\a} + \tilde{\n}_2$. Therefore, we have that $[\varphi(X_B), \tilde{X}_C] = \phi(\tilde{X}_B) \tilde{X}_{JC}$. Nevertheless, if we take $\tilde{X}_C = \tilde{X}_B$, then for all $\tilde{X}_B \in \tilde{\n}_1$, $\phi(\tilde{X}_B) \tilde{X}_{JB} = 0$. This last equation implies $\phi(\tilde{X}_B) = 0$ and $\varphi(X_B) = 0$, for all $X_B \in \n_1$. Then necessarily the holonomy algebra $\mathfrak{hol} = \varphi (\n)$ is trivial or a one dimensional subspace generated by $\curv_{\tilde{A}_0 \tilde{N}_2} = \varphi (N_2)$.
	
	We take $B = \tilde{N}_2$ and $B = \tilde{A}_0$ in \eqref{Equation 6.2} we have
	\begin{equation} \label{Equation 6.3}
		g([\varphi(N_2), \tilde{X}_C], \tilde{X}_D) = \phi(\tilde{N}_2) g(\tilde{X}_{JC}, \tilde{X}_D)
		\end{equation}
		\begin{equation} \label{Equatkion 6.4}
		g([\varphi(A_0) + H_r, \tilde{X}_C], \tilde{X}_D) = \phi(\tilde{A}_0) g(\tilde{X}_{JC}, \tilde{X}_D),
	\end{equation}
	respectively. From these two equations, we have $ [\varphi(N_2), \tilde{X}_C] =  \phi(\tilde{N}_2) \tilde{X}_{JC}$ and $[\varphi(A_0) + H_r, \tilde{X}_C] = \phi(\tilde{A}_0) \tilde{X}_{JC}$.
	Now we consider the two cases above, first, if $\mathfrak{hol}$ is zero, then $\varphi(N_2) = 0$ and  $[\varphi(A_0)+ H_r, \tilde{X}_C] = \beta J\tilde{X}_B$ with $\beta = \phi(\tilde{A}_0)$. Secondly, if $\mathfrak{hol}$ is non zero, then $\varphi(N_2) \neq 0$ and  $[\varphi(N_2),\tilde{X}_C] = \lambda J\tilde{X}_B$ with $\lambda = \phi(\tilde{N}_2)$. Moreover, in this case, necessarily $H_r = 0$. As $\varphi(A_0), \varphi(N_2) \in \h$, $H_r \notin \h$ (beginning of Section \ref{Section 3}) and because \eqref{Equation 6.3} and \eqref{Equatkion 6.4}, $H_r = \frac{\phi(\tilde{A}_0)}{ \phi(\tilde{N}_2)} \varphi(N_2) - \varphi(A_0) = 0$.
	
	Conversely, it is direct to check that taking these two cases $S$ belongs to $\mathcal{L}$.
\end{proof}

\begin{remark}\label{Remark 6.7}
	The subspace $\mathcal{K}_{2+4}(V)\cap\ker(c_{12})$ has the expression,
	\begin{align*}
		& \{S \in \mathcal{K}(V) :\: S_{BCD} = g(B,C)\gamma(D) -  g(B,D) \gamma(C)+  g(B,JC) \gamma(JD)  \\
		& \qquad \qquad - g (B,JD)\gamma(JC) +2n g(JC,D) \gamma(JB),\, \gamma \in V^*\}.
	\end{align*}
\end{remark}

\begin{lemma}\label{lemurcillo}
	Let $S$ be a K\"ahler homogeneous structure on $\CH (n)=G/H$ and $\mathfrak{hol}$ its holonomy algebra. Then, $S$ belongs to $\mathcal{K}_{1+2+4}(V)$ if and only if $S$ belongs to $\mathcal{K}_{2+4}(V)$.
\end{lemma}

\begin{proof}
	Let $S=E_1 + E_2 + E_3 + E_4$ be a homogeneous structure in $\mathcal{K}_{1+2+4}(V)$. As $E_1 \in \ker (c_{12}) ^{\perp}$ and $E_2+E_3+E_4 \in \ker (c_{12})$, then $S \in \mathcal{K}_{1+2+4}(V)$ if and only if $E_2+E_3+E_4 \in \mathcal{K}_1(V) + \mathcal{K}_{2+4}(V)\cap\ker(c_{12})$. Equivalently, because of Remark \ref{Remark 6.7}, there exists $\gamma \in V^*$ and $S^1 \in \mathcal{K}_1(V)$ such that,
	\begin{equation}\label{Equation 6.5}
		\begin{alignedat}{3}
			\alpha_{JB} &g(\tilde{X}_{JC}, \tilde{X}_D) + g([B',C],D) + \alpha_B g(C_r,D)= g(B,C) \gamma(D) \\
			&- g(B,D)\gamma(C)	 + g(B,JC)\gamma(JD) - g(B,JD)\gamma(JC) \\
			& +2n g(JC,D)\gamma(JB) + S^1_{BCD}.
		\end{alignedat}
	\end{equation}
	As $S^1 \in \mathcal{K}_1(V)$ (see \eqref{Equation 2.3}), then
	\begin{equation}\label{Equation 6.6}
		S^1_{B \tilde{N}_2 \tilde{A}_0} = -S^1_{\tilde{A}_0 \tilde{N}_2 B}+ S^1_{ \tilde{N}_2 \tilde{A}_0 B}
	\end{equation}
	for all $B\in \m$. Now we proceed by parts. Firstly, by substituting $B\in \tilde{\n}_1$, $C=\tilde{N}_2$ and $D = \tilde{A}_0$ in \eqref{Equation 6.5}, we get
	\begin{equation*}
			n \mu \gamma(JB) +  S^1_{B \tilde{N}_2 \tilde{A}_0} = 0.
	\end{equation*}
	Secondly, by substituting $B = \tilde{A}_0$, $C=\tilde{N}_2$ and $D \in \tilde{\n}_1$ in \eqref{Equation 6.5}, we get
	\begin{equation*}
		\frac{\mu}{2} \gamma(JD) +  S^1_{\tilde{A}_0 \tilde{N}_2 D} = 0.
	\end{equation*}
	By substituting $B = \tilde{N}_2$, $C=\tilde{A}_0$ and $D \in \tilde{\n}_1$ in \eqref{Equation 6.5}, we get
	\begin{equation*}
		\frac{-\mu}{2} \gamma(JD) +  S^1_{\tilde{N}_2 \tilde{A}_0  D} = 0.
	\end{equation*}
	Finally, taking in consideration these equations in \eqref{Equation 6.6}, we obtain,
	\begin{equation*}
		n \mu \gamma(JD) = - \mu \gamma(JD).
	\end{equation*}
	Therefore, $\gamma (JD) = 0$ for any $D \in \n_1$. Now we claim that $\gamma(\tilde{N}_2) = 0$ and $\gamma (\tilde{A}_0) = 0$. From \eqref{Equation 6.6}, we get that $S^1_{\tilde{A}_0 \tilde{N}_2 \tilde{A}_0} = 0$ and by substituting in \eqref{Equation 6.5}, we conclude that $\gamma(\tilde{N}_2) = 0$. Arguing analogously, we prove that $\gamma(\tilde{A}_0) = 0$. This proves that $\gamma = 0$ and $E_2 + E_3 + E_4 \in \mathcal{K}_1(V)$. As $(E_2)_{BCD}$ belongs to $\mathcal{K}_{2+3+4}(V)$ and is non-zero if and only if $B$ is co-linear with $\tilde{N}_2$, then it is necessary that $(E_2 + E_3 + E_4)_{\tilde{N}_2 C D} = 0$ for any $B,C\in \m$. We compute this and get
	\begin{equation*}
		\frac{1}{2} g(\tilde{X}_{JC}, \tilde{X}_D) + g([\varphi(N_2), X_C], X_D) = 0.
	\end{equation*}
	Then, $[\varphi(N_2), X_C] = - \frac{1}{2} J X_C$. Therefore, $\varphi(N_2)$ acts effectively in all $\n_1$ and because of Lemmas \ref{Lemma 5.1} and \ref{Lemma 5.2}, $\mathfrak{hol}$ is generated by $\varphi(N_2)$. Indeed, this implies that $g([B', C], D) = - \alpha_{JB}g(\tilde{X}_{JC}, \tilde{X}_D)$. Therefore, $\alpha_B g(C_r, D) \in \mathcal{K}_1(V)$, taking in account the identity of $\mathcal{K}_1(V)$ (see \eqref{Equation 2.3}). We finish that $\alpha_B g(C_r, D) = 0$ and $E_2+E_3+E_4 = 0$.	
\end{proof}

\begin{theorem} \label{Teorema final sección 6}
	Let $S$ be a K\"ahler homogeneous structure on $\CH (n)$ and $\mathfrak{hol}$ its holonomy algebra. Then,
	\begin{enumerate}
		\item $S=0$ if and only if $\mathfrak{hol} = \s (\u(n) + \u(1))$.
		\item $S$ belongs strictly to $\mathcal{K}_{2+4}(V)$ if and only if $\mathfrak{hol}$ is one dimensional with $\curv_{\tilde{A}_0 \tilde{N}_2} \tilde{X}=  J \tilde{X}$, for $\tilde{X} \in \tilde{\n}_1$, and $\a_r = \a \subset \ker \varphi$.
		\item  $S$ belongs strictly to $\mathcal{K}_{2+3+4}(V)$ if and only if  happens one of the following two:
		\begin{itemize}
			\item The holonomy algebra  $\mathfrak{hol}$ is one dimensional with $ \curv_{\tilde{A}_0 \tilde{N}_2} \tilde{X}=  \lambda J \tilde{X}$, for $\tilde{X} \in \tilde{\n}_1$ $\lambda \neq 1$, $\a_r = \a$ and $\varphi(\a) = \varphi(\n_2) = \mathfrak{hol}$.
			\item The holonomy algebra  $\mathfrak{hol}$ is trivial and $[\varphi(A_0)+ H_r, \tilde{X}] = \beta J\tilde{X}$, for $\tilde{X} \in \tilde{\n}_1$ where $H_r= A_r - A_0$ and $\beta \in \R$.
		\end{itemize}
		\item Otherwise, $S$ is of general type.
	\end{enumerate}
\end{theorem}

\begin{proof}
	We proceed by parts.
	
	We prove the first asseveration. Let $S$ be a homogeneous structure of $\CH(n)$, $S$ is equal to zero if and only if the Levi-Civita connection coincides with the canonical connection. As the holonomy algebra of the Levi-Civita connection coincides with $\mathfrak{hol}$, then $\mathfrak{hol} = \s (\u(n) + \u(1))$.
	
	We prove the second asseveration. Consider the decomposition $\mathcal{K}_{2+4}(V) = \ker(c_{12})^{\perp} + \mathcal{K}_{2+4}(V)\cap\ker(c_{12})$.	As $E_2, E_3, E_4$ are orthogonal to $\ker(c_{12})^{\perp}$ and $E_1 \in \ker(c_{12})^{\perp}$, then $S \in \mathcal{K}_{2+4}(V)$ if and only if $E_2 + E_3 + E_4 \in \mathcal{K}_{2+4}(V)\cap\ker(c_{12})$. Because of Remark \ref{Remark 6.7}, then these exists a $\gamma \in V^*$ such that for every $B,C,D \in \m$,
	\begin{align*}
		\alpha_{JB}g(\tilde{X}_{JC}, \tilde{X}_D) &+ g([B',C],D) + \alpha_B g(C_r,D) =\\
		& g(B,C) \gamma(D) - g(B,D)\gamma(C)+ g(B,JC)\gamma(JD)\\
		&- g(B,JD)\gamma(JC) +2n g(JC,D)\gamma(JB).		
	\end{align*}
	If we take in this equation above $C=\tilde{N}_2$ and $D = \tilde{A}_0$. Then, as $[\s(\u(n-1)+\u(1)), \tilde{N}_2] = 0$ the first row is zero. In the second row we use the identity, $J\tilde{N}_2 = \frac{1}{2} \tilde{A}_0$ and $J\tilde{A}_0 = -2 \tilde{N}_2$. Therefore,
	\begin{equation*}
		0 = 2 g(B, \tilde{N}_2) \gamma(\tilde{A}_0) - 2 g(B, \tilde{A}_0) \gamma(\tilde{N}_2) + n g(\tilde{A}_0, \tilde{A}_0)\gamma(JB).
	\end{equation*}
	Now, $g(\tilde{A}_0, \tilde{A}_0) = \mu$ and solving for $\gamma(JB)$,
	\begin{equation*}
		\gamma(JB) = \frac{2 g(B, \tilde{A}_0) \gamma(\tilde{N}_2) - 2 g(B, \tilde{N}_2) \gamma(\tilde{A}_0)}{n \mu}.
	\end{equation*}
	Finally, we consider the following cases: if $B \in \tilde{\n}_1$, then $\gamma(JB) = 0$; else, if $B= \tilde{A}_0$ (recall that $J \tilde{A}_0 = -2 \tilde{N}_2$), then $-2 \gamma(\tilde{N}_2) = \frac{2}{n} \gamma(\tilde{N}_2)$; else, if $B= \tilde{N}_2$, then $\frac{1}{2}\gamma(\tilde{A}_0) = \frac{-1}{2n} \gamma(\tilde{N}_2)$. Consequently, if $n \neq -1$, then $\gamma = 0$ and this case is reduced to Lemma \ref{lemma 6.5}.
	
	We prove the third asseveration. Consider the decomposition $\mathcal{K}_{2+3+4}(V) = \mathcal{K}_{2+4}(V)\cap\ker(c_{12})^{\perp} + \mathcal{K}_{2+4}(V)\cap\ker(c_{12})+ \mathcal{K}_3(V)$. As $E_1, E_2 \in \mathcal{K}_{2+3+4}(V)$ and $E_3, E_4$ are orthogonal to $\mathcal{K}_{2+4}(V)\cap\ker(c_{12})^{\perp}$, then, $S \in \mathcal{K}_{2+3+4}(V)$ if and only if $E_3 +E_4 \in \mathcal{K}_{2+4}(V)\cap\ker(c_{12})+ \mathcal{K}_3(V)$. Therefore, there exists a $\gamma \in V^*$ and $S^3 \in \mathcal{K}_3(V)$ such that for every $B,C,D \in \m$,
	\begin{equation}\label{Equation 6.7}
		\begin{alignedat}{3}
			g([B',C],D) &+ \alpha_B g(C_r,D)= g(B,C) \gamma(D) - g(B,D)\gamma(C) \\
			& + g(B,JC)\gamma(JD) - g(B,JD)\gamma(JC) \\
			& +2n g(JC,D)\gamma(JB) + S^3_{BCD}.
		\end{alignedat}
	\end{equation}
	As $S^3 \in \mathcal{K}_3(V)$, then
	\begin{equation*}
		S^3_{B \tilde{N}_2 \tilde{A}_0} = S^3_{\tilde{A}_0 \tilde{N}_2 B}- S^3_{ \tilde{N}_2 \tilde{A}_0 B}
	\end{equation*}
	for all $B \in \m$. Arguing analogously as in Lemma \ref{lemurcillo} we obtain,
	\begin{equation*}
		(n-1)\mu \gamma (JD) = 0
	\end{equation*}
	which implies that $\gamma (JD) = 0$ for all $D\in \tilde{\n}_1$ (remind that if $n=1$, then $S=E_1$, indeed $\gamma(JD)$ is necessarily zero). Therefore, $\gamma \in (\tilde{\a} + \tilde{\n}_2 )^*$. Equivalently to \eqref{Equation 6.7},  we can consider another $\hat{\gamma}$ colinear with $\gamma$ and another $\hat{S}^3 \in \mathcal{K}_3(V)$ such that,
	\begin{equation}\label{Equation 6.8}
		g([B',C],D) + \alpha_B g(C_r,D)= \hat{\gamma}(JB) g(\tilde{X}_{JC},\tilde{X}_D) + \hat{S}^3_{BCD}.
	\end{equation}
	The tensor element $ \hat{\gamma}(JB) g(\tilde{X}_{JC},\tilde{X}_D)$ belongs to $\mathcal{K}_{2+3+4}(V)$ with projection to $\mathcal{K}_{2+4}(V)$ is equal to $g(B,C) \gamma(D) - g(B,D)\gamma(C) + g(B,JC)\gamma(JD) - g(B,JD)\gamma(JC) +2n g(JC,D)\gamma(JB)$.	
	
	We claim $\hat{S} = 0$. By taking $B, C, D \in \tilde{\n}_1$ in \eqref{Equation 6.8} and $\hat{\gamma}(\tilde{\n}_1) = 0$,
	\begin{equation*}
		\hat{S}^3_{BCD} = g([B',C], D),
	\end{equation*}
	then, $g([B',C], D)$ satisfies $\mathcal{K}_3(V)$ identity (see \eqref{Equation 2.2}), that is,
	\begin{equation}\label{Equation 6.9}
		\begin{alignedat}{2}
			g([B',C], D) = &- \frac{1}{2} g([C',D], B) - \frac{1}{2} g([D',B],C) \\
			& - \frac{1}{2} g([(JC)',JD], B) - \frac{1}{2} g([(JD)',B],JC).
		\end{alignedat}
	\end{equation}
	Firstly, we consider the decomposition $\n_1 = V_{ss} + V_0 + V'$ as in \eqref{orthogonal decomposition}. Because of Lemmas \ref{Lemma 5.1} and \ref{Lemma 5.2}, $\k_0$ acts effectively on $V'$. Consequently, if $X_B \in V_0 + V'$, then $\hat{S}^3_{BCD} = 0$. Therefore, from now, we consider $X_B\in V_{ss}$. Secondly, by substituting $D= [B', C]$ in \eqref{Equation 6.9} and using the identities: $[B', B] =0$, the Jacobi identity and $\hat{S}^3 \cdot g = 0$. Then,
	\begin{equation}\label{Equation 6.10}
		\big | \big | [B', C] \big |\big |^2 = g([C',B], [B',C]) + g([(JC)',B], [B', JC]).
	\end{equation}
	We decompose, $X_C = (X_C)_{ss} + (X_C)_0 + (X_C)_{V'}$ and $X_{JC} = (X_{JC})_{ss} + (X_{JC})_0 + (X_{JC})_{V'}$ with each sum belonging to each sum of $\n_1 = V_{ss} + V_0 + V'$, respectively. As $\k_0$ acts trivially on $V_{ss}$ and $X_B\in V_{ss}$, then $[\varphi(X_C), X_B] = [\varphi((X_C)_{ss}), X_B]$ and $[\varphi(X_{JC}),X_B] = [\varphi((X_{JC})_{ss}),X_B]$. Therefore, because of $\ker (\varphi|_{V_{ss}}) = \{0\}$, then $[\varphi((X_C)_{ss}), X_B] = - [\varphi(X_B), (X_C)_{ss}] = - [\varphi(X_B), X_C - (X_C)_{V'}]$ and $[\varphi((X_{JC})_{ss}), X_B] = - [\varphi(X_B), (X_{JC}) - (X_{JC})_{V'}]$. Taking tildes  and substituting these in \eqref{Equation 6.10},
	\begin{equation*}
		\begin{alignedat}{2}
			3 \big | \big | [B', C] \big |\big |^2 &= \big | \big | [B', (\tilde{X}_C)_{V'}] \big |\big |^2 + \big | \big | [B', (\tilde{X}_{JC})_{V'}] \big |\big |^2\\
			& \leq 2 \big | \big | [B', C] \big |\big |^2
		\end{alignedat}		
	\end{equation*}
	Therefore, $\hat{S}^3_{BCD} = 0$ for $B,C,D \in \tilde{\n}_1$. From \eqref{Equation 6.8}, for any $B, C\in \m$, $\hat{S}^3_{B C \tilde{A}_0} = 0$ and $\hat{S}^3_{B C \tilde{N}_2} = 0$. Finally, using \eqref{Equation 6.9}, we get $\hat{S}^3_{\tilde{A}_0 B C} = \hat{S}^3_{\tilde{N}_2 B C} = 0$ and consequently $\hat{S} = 0$. Indeed, \eqref{Equation 6.8} has de expression,
	\begin{equation*}
		g([B',C],D) + \alpha_B g(C_r,D)= \hat{\gamma}(JB) g(\tilde{X}_{JC},\tilde{X}_D)
	\end{equation*}
	and we are in condition of Lemma \ref{lemma 6.6}.
		
	Otherwise, as any homogeneous structure $S$ has a non-zero part $E_1$ in $\ker (c_{12})^{\perp} \subset \mathcal{K}_{2+4}(V)$ and $E_2, E_3, E_4 \in \ker (c_{12})$. Then, the two missing cases to study are $S$ belongs strictly to $\mathcal{K}_{1+2+4}(V)$ or $S$ is of general type. Because of Lemma \ref{lemurcillo}, $S$ must be of general type.	
\end{proof}

\paragraph{Acknowledgments.}

Both authors have been partially supported by grant no. PGC2018-098321-B-I00, Ministerio de Ciencia e Innovación, Spain. MCL has been partially supported by grant no. SA090G19, Consejería de Educación, Junta de Castilla y León,Spain.

\addcontentsline{toc}{section}{References}

\nocite{*} 


\end{document}